\documentclass[12pt, reqno]{amsart}
\usepackage{amsmath, amsthm, amscd, amsfonts, amssymb, graphicx, color}
\usepackage[bookmarksnumbered, colorlinks, plainpages]{hyperref}
\hypersetup{colorlinks=true,linkcolor=red, anchorcolor=green, citecolor=cyan, urlcolor=red, filecolor=magenta, pdftoolbar=true}

\usepackage{tensor}

\newtheorem{theorem}{Theorem}[section]
\newtheorem{lemma}[theorem]{Lemma}
\newtheorem{proposition}[theorem]{Proposition}

\theoremstyle{definition}
\newtheorem{definition}[theorem]{Definition}
\newtheorem{example}[theorem]{Example}

\theoremstyle{remark}
\newtheorem{remark}[theorem]{Remark}
\numberwithin{equation}{section}

\usepackage{fullpage}

\begin{document}

\setcounter{page}{1}

\title[FPT on $G$-metric spaces]{
	$n$-tuple fixed point in $\phi$-ordered $G$-metric spaces}

\author[Ya\'e Ulrich Gaba]{Ya\'e Ulrich Gaba$^{1,2}$}

\author{Collins Amburo Agyingi$^{2}$}

\email{\textcolor[rgb]{0.00,0.00,0.84}{gabayae2@gmail.com
}}

\address{$^{1}$\'Ecole Normale Sup\'erieure (ENS) de Natitingou, BP 72 Natitingou, B\'enin.}

\address{$^{2}$ Department of Mathematical Sciences, North West University, Private Bag
	X2046, Mmabatho 2735, South Africa.}

\email{\textcolor[rgb]{0.00,0.00,0.84}{: collins.agyingi@nwu.ac.za
}}

\subjclass[2010]{Primary 47H05; Secondary 47H09, 47H10.}

\keywords{$G$-metric; $\phi$-order;
	weakly related; $n$-tuple fixed point fixed point.}

\begin{abstract}
	We use three seminal approaches in the study of fixed point theory, the so called $G$-metrics, multidimensional fixed points and partially ordered spaces.
	More precisely, we extend known results from the theory of quasi-pseudometric spaces to the $G$-metric space setting. In particular, we show the existence of $n$-tuple fixed points (resp. common $n$-tuple fixed point) for a non-decreasing mapping (resp. a pair of weakly related mappings) in a $\phi$-ordered $G$-metric space.

\end{abstract} 

\maketitle

\section{Introduction and preliminary results.}
Metric spaces have been extensively used to solve major problems appearing in quantitative sciences and considering various
generalizations of metrics and metric spaces  is a natural step in order to broaden the scope of applied sciences. In
this regard, $G$-metric spaces ($G_b$-metric spaces), cone metric spaces and quasi-pseudometric
spaces are relevant instances. Most of these applications are done via fixed point theory whose relevance is no more to be demonstrated as it has been extensively discussed in many divisions of applied sciences.

\vspace*{0.2cm}

Recently, the study of multidimensional fixed point has been at the center of very active research, see \cite{ref3,ref4,ref2, ref1}. The results obtained appear both in metric spaces and in generalized metric spaces, of which $G$-metrics is our space of focus here (see \cite{Mustafa}).

In this article, we replace the left $K$-complete quasi-pseudometric spaces $(X, d)$ by $\phi$-ordered complete $G$-metric spaces $(X,G)$ and prove some fixed point theorems in that setting. Our
results generalized some fixed point theorems in quasi-pseudometric spaces \cite{gaba,gaba1}. Namely, we study the notion of coupled (resp. $n$-tuple) fixed points in the setting of $G$-metric spaces endowed with a partial order.
As we mentioned in \cite{gaba} (and it is the same in the present manuscript), the technique of proof employed differs from the classical one and is more natural in the sense that we do not utilize any contractive conditions. Furthermore, our approach does not differentiate between $G$-metric spaces and $G_b$-metric spaces.

\vspace*{0.2cm}

Throughout this manuscript, $X$ will be a non-empty set and $\preceq$ will denote a preorder on $X$ induced by a certain function $\phi$. Given $n \in \mathbb{N}$ with $n\geq 2$, $X^n$ will denote the product space $X \times X \times \ldots  X$ of $n$ identical copies of $X$. In the next few lines, we recall some concepts and fix our notations. We shall only recall the necessary notions. The interested reader is referred to \cite{gaba,Mustafa} for a more detailed expos\'e. The results presented are a  generalization of a previous work by Gaba\cite{gaba}.

\begin{definition}(Compare \cite{gaba})
	Let $(X,\preceq_{{}_X})$ and $(Y,\preceq_{{}_Y})$ be two prosets. A map $T:X\to Y$ is said to be \textbf{preorder-preserving} or \textbf{isotone} if for any $x,y \in X, $
	\[
	x\preceq_{{}_X} y \Longrightarrow Tx \preceq_{{}_Y} Ty.
	\]
	Similarly, for any family $(X_i,\preceq_{{}_{X_i}}),\ i=1,2,\cdots,n; \ (Y,\preceq_{{}_Y})$ of posets, a mapping $F:X_1\times X_2\times \cdots\times X_n \to Y$ is said to be \textbf{preorder-preserving} or \textbf{isotone} if for any 
	for any  $(x_1,x_2,\cdots,x_n), (z_1,z_2,\cdots,z_n) \in X_1\times X_2\times \cdots\times X_n, $
	\[
	\quad x_i\preceq_{{}_{X_i}} z_i \text{ for all } i=1,2,\cdots,n \Longrightarrow F(x_1,x_2,\cdots,x_n) \preceq_{{}_{Y}} F(z_1,z_2,\cdots,z_n).
	\]
\end{definition}

Next we recall the basic concepts and notations attached to the idea of $G$-metric. This can be read extensively in \cite{Mustafa}.

\begin{definition}\label{def1} (See \cite[Definition 3]{Mustafa})
	Let $X$ be a nonempty set, and let the function $G:X\times X\times X \to [0,\infty)$ satisfy the following properties:
	\begin{itemize}
		\item[(G1)] $G(x,y,z)=0$ if $x=y=z$ whenever $x,y,z\in X$;
		\item[(G2)] $G(x,x,y)>0$ whenever $x,y\in X$ with $x\neq y$;
		\item[(G3)] $G(x,x,y)\leq G(x,y,z) $ whenever $x,y,z\in X$ with $z\neq y$;
		\item[(G4)] $G(x,y,z)= G(x,z,y)=G(y,z,x)=\ldots$, (symmetry in all three variables);
		
		\item[(G5)]
		$$G(x,y,z) \leq [G(x,a,a)+G(a,y,z)]$$ for any points $x,y,z,a\in X$.
	\end{itemize}
	Then $(X,G)$ is called a \textbf{$G$-metric space}.
	
\end{definition}

\begin{definition}(See \cite{Mustafa})
	Let $(X, G)$ be a $G$-metric space, and let $\{x_n \}$ be a
	sequence of points of $X$, therefore, we say that $(x_n )$ is $G$-convergent to
	$x \in X$ if $\lim_{n,m\to \infty} G (x, x_n , x_m ) = 0,$ that is, for any $\varepsilon > 0$, there exists $N \in \mathbb{N}$ such that $G (x, x_n , x_ m ) < \varepsilon$, for all, $n, m \geq N$. We call $x$ the limit of the sequence and write $x_n \to x$ or $\lim_{n\to \infty} x_n = x$.
	
\end{definition}

\vspace*{0.2cm}

\begin{proposition}\label{prop1} (Compare \cite[Proposition 6]{Mustafa})
	Let $(X,G)$ be a $G$-metric space. Define on $X$ the metric  $d_G$ by $d_G(x,y)= G(x,y,y)+G(x,x,y)$ whenever $x,y \in X$. Then for a sequence $(x_n) \subseteq X$, the following are equivalent
	\begin{itemize}
		\item[(i)] $(x_n)$ is $G$-convergent to $x\in X.$
		
		\item[(ii)] $\lim_{n,m \to \infty}G(x,x_n,x_m)=0.$

		\item[(iii)]  $\lim_{n \to \infty}d_G(x_n,x)=0$.

		\item[(iv)]$\lim_{n \to \infty}G(x,x_n,x_n)=0.$ 
		
		\item[(v)]$\lim_{n \to \infty}G(x_n,x,x)=0.$ 
	\end{itemize}
	
\end{proposition}

\vspace*{0.2cm}

\begin{definition}(See \cite{Mustafa})
	Let $(X, G)$ be a $G$-metric space. A sequence $\{x_n \}$ is
	called a $G$-Cauchy sequence if for any $\varepsilon > 0$, there is $N \in \mathbb{N}$ such that
	$G (x_n , x_m , x_l ) < \varepsilon$ for all $n, m, l \geq N$, that is $G (x_n , x_m , x_l ) \to 0$ as $n, m,l \to +\infty$.
	We shall use interchangeably ``$G$-Cauchy sequence in a $G$-metric space"  or ``Cauchy sequence in a $G$-metric space".
\end{definition}

\vspace*{0.2cm}

\begin{proposition}(Compare \cite[Proposition 9]{Mustafa})
	
	In a $G$-metric space $(X,G)$, the following are equivalent
	\begin{itemize}
		\item[(i)] The sequence $(x_n) \subseteq X$ is $G$-Cauchy.
		
		\item[(ii)] For each $\varepsilon >0$ there exists $N \in \mathbb{N}$ such that $G(x_n,x_m,x_m)< \varepsilon$ for all $m,n\geq N$.
		
	\end{itemize}
	
\end{proposition}

\vspace*{0.2cm}

\begin{definition} (Compare \cite[Definition 9]{Mustafa})
	A $G$-metric space $(X,G)$ is $G$-complete if every $G$-Cauchy sequence of elements of $(X,G)$ is $G$-convergent in  $(X,G)$. We shall use interchangeably ``$G$-complete $G$-metric space" or ``complete $G$-metric space".
	
\end{definition}

\begin{theorem}\label{theorem1}(See \cite{Mustafa})
	A $G$-metric $G$ on a $G$-metric space $(X, G)$ is continuous on its three
	variables.
\end{theorem}

\begin{definition}
	Let $(X,G)$ be a $G$-metric space. A function $T:X\to X$ is called \textbf{sequentially continuous} if for any $G$-convergent sequence $(x_n)$ with $x_n\longrightarrow x $, the sequence $(Tx_n)$ $G$-converges to $Tx$, i.e. $Tx_n \longrightarrow Tx $.
	
	Similarly, a function $T:X_1\times X_2\times \cdots \times X_n \to X$ for $n\geq 2,$ is said to be \textbf{sequentially continuous} if for any sequences $(x_l^i)$ such that $x_l^i \longrightarrow x^{*,i} ,$ then $$T(x_l^i,x_l^{i+1},\cdots,x_l^n,x_l^i,\cdots,x_l^{i-1})\longrightarrow T(x^{*,i},x^{*,i+1},\cdots,x^{*,n},x^{*,1},\cdots,x^{*,i-1}).$$
	
	\end{definition}

\begin{definition}(Compare \cite{er})
	An element $(x^1,x^2,\cdots,x^n) \in X^n$ is called:
	\begin{enumerate}
		\item[(E1)] a \textbf{$n$-tuple fixed point} of the mapping $F:X^n \to X$ if
		$$F(x^i,x^{i+1},\cdots,x^n, x^1,\cdots,x^{i-1})=x^i, \text{ for all } i, 1\leq i\leq n.  $$
		
		\item[(E2)] a \textbf{$n$-tuple coincidence point} of the mappings $F:X^n \to X$ and $T:X \to X$ if $F(x^i,x^{i+1},\cdots,x^n, x^1,\cdots,x^{i-1})=Tx^i$ for all $i, 1\leq i\leq n$ and in this case $(Tx^1,Tx^2,\cdots,Tx^n)$ is called the \textbf{$n$-tuple point of coincidence};
		\item[(E3)] a \textbf{common $n$-tuple fixed point} of the mappings $F:X^n \to X$ and $T:X \to X$ if $F(x^i,x^{i+1},\cdots,x^n, x^1,\cdots,x^{i-1})=Tx^i=x^i$ for all $i, 1\leq i\leq n$.
	\end{enumerate}
\end{definition}

\begin{definition}(Compare \cite{gaba})
	
	An element $(x^1,x^2,\cdots,x^n) \in X^n$ is called
	
	\begin{enumerate}
		\item[(D1)] a \textbf{$n$-tuple coincidence point} of the mappings $F:X^n \to X$ and $T,R:X \to X$ if $F(x^i,x^{i+1},\cdots,x^n, x^1,\cdots,x^{i-1})=Tx^i=Rx^i$ for all $i, 1\leq i\leq n$;
		
		\item[(D2)] a \textbf{common $n$-tuple fixed point} of the mappings $F:X^n \to X$ and $T,R:X \to X$ if $F(x^i,x^{i+1},\cdots,x^n, x^1,\cdots,x^{i-1})=Tx^i=Rx^i=x^i$ for all $i, 1\leq i\leq n$.
	\end{enumerate}
	
\end{definition}

\section{$n$-tuple fixed point}
We first prove the following lemma:

\begin{lemma}(Compare \cite[Lemma 3.1.]{gaba})
	Let $(X,G)$ be a $G$-metric space and $\phi : X \to \mathbb{R}$ a map. Define the binary relation $"\preceq"$ on $X$ as follows:
	\[
	x \preceq y \Longleftrightarrow G(x,y,y) \leq \phi(y)-\phi(x).
	\]
	Then $"\preceq"$ is a preorder on $X$. It will be called the preorder induced by $\phi$.
\end{lemma}

\begin{proof} \hspace*{0.2cm}
	
	\begin{itemize}
		\item Reflexivity: For all $x \in X$; 
		$$0=G(x,x,x) \leq \phi(x)-\phi(x)=0,$$
		hence $x\preceq x$, i.e., "$\preceq$" is reflexive.
		
		\item Transitivity: For $x, y, z \in X$ s.t. $ x\preceq y$ and $y\preceq z$, we
		have
		\[  G(x,y,y) \leq \phi(y)-\phi(x) \text{ and }  G(y,z,z) \leq \phi(z)-\phi(y).\]
		By property (G5), we have
		\[G(x,z,z) \leq G(x,y,y)+G(y,z,z)\leq \phi(y)-\phi(x) + \phi(z)-\phi(y) = \phi(z)-\phi(x),  \]
		i.e. $x\preceq z.$ Thus, "$\preceq$" is transitive, and so the
		relation "$\preceq$" is a preorder on $X$.
	\end{itemize}
\end{proof}

\begin{example}
	Let $X = [0, \infty)$ and $G(x, y,z) = \max\{x,y,z\}$, then $(X, G)$ is a
	$G$-metric space. Let $\phi : X \to \mathbb{R},\ \phi(x) = 2x.$ Then for $x, y \in X$
	
	\begin{align*}
	x\preceq_\phi y:=	x\preceq y & \Longleftrightarrow G(x,y,y) \leq \phi(y) - \phi(x) \\
		         & \Longleftrightarrow \max\{x,y\} \leq 2y - 2x. 
	\end{align*}
	It follows that $$2 \preceq 4, \ \ \ \  \frac{1}{4} \preceq \frac{1}{2}, \ \  \ \  2 \preceq 2,$$ whereas $3$ is not comparable
	to $2$ and $6$ is not comparable to $5$, etc. Therefore $X$ is a $\phi$-ordered $G$-metric space. 
\end{example}

\begin{remark}
Note that $(X,\preceq_\phi)$ is just a preordered space in general. However, if the $G$-metric $G$ is symmetric, i.e. $G(x,y,y)=G(y,x,x)$, then $(X,\preceq_\phi)$ is an ordered space.
\end{remark}

Now we prove the following theorem.

\begin{theorem}
	Let $(X, G)$ be a complete $G$-metric space, $\phi : X \rightarrow \mathbb{R}$ be a bounded from above function and ``$\preceq$" the preorder induced by $\phi$. Let $F : X^n \rightarrow X,$ $n\geq 2$ be a preorder preserving and sequentilally continuous mapping on $X^n$ such that there exist $n$ elements $x_0^1,\cdots,x_0^n \in X$ with 
	
	\[
	x_0^i \preceq F(x_0^i,x_0^{i+1},\cdots,x_0^n,x_0^1,\cdots,x_0^{i-1})\  \text{  for all } i, \ 1\leq i \leq n.
	\]
	Then $F$ has a $n$-tuple fixed point in $X^n$.
	
\end{theorem}

\begin{proof}
	Let $x_0^1,\cdots, x_0^n\in X$ with 
	\[
	x_0^i \preceq F(x_0^i,x_0^{i+1},\cdots,x_0^n,x_0^1,\cdots,x_0^{i-1})\ \ \text{  for all } i, \ 1\leq i \leq n.
	\]
	
	We construct the sequences $(x_l^i)_l$ for $1\leq i \leq n$ as follows:
	
	\begin{equation}\label{eq2}
	x_{l+1}^i = F(x_l^i,x_l^{i+1},\cdots,x_l^n,x_l^1,\cdots,x_l^{i-1}), \ \  \text{  for all } i, \ 1\leq i \leq n.
	\end{equation}
	We shall show that 
	\begin{equation}\label{eq3}
	x_l^i \preceq x_{l+1}^i \ \ \  \text{  for all } l\geq 0.
	\end{equation}
	
	We use mathematical induction. Since $x_0^i \preceq F(x_0^i,x_0^{i+1},\cdots,x_0^n,x_0^1,\cdots,x_0^{i-1})$, we have $x_0^i \preceq x_1^i$. Thus \eqref{eq3} holds for $l=0$. Suppose that \eqref{eq3} holds for some $k>0$. Then since $x^i_k \preceq x^i_{k+1}$ and $F$ is preorder preserving, we have

	\begin{align}
	x^i_{k+1}& = F(x_k^i,x_k^{i+1},\cdots,x_k^n,x_k^1,\cdots,x_k^{i-1}) \\
	& \preceq   F(x_{k+1}^i,x_{k+1}^{i+1},\cdots,x_{k+1}^n,x_{k+1}^1,\cdots,x_{k+1}^{i-1})        \\
	&=  x^i_{k+2}
	\end{align}
	
	Thus by mathematical induction we conclude that \eqref{eq3} holds for all $l \geq 0$. Therefore 
	\[
	x_0^i \preceq x_1^i \preceq x_2^i \preceq x_3^i \preceq \cdots \preceq x_l^i \preceq \cdots .
	\]
	
	By definition of the preorder, we have
	
	\[
	\phi(x_0^i) \leq \phi(x_1^i) \leq \phi(x_2^i) \leq \phi(x_3^i) \leq \cdots \leq \phi(x_l^i) \leq \cdots .
	\]
	
	Hence, the sequence $(\phi(x^i_l ))$ is a non-decreasing sequence of real numbers. Since $\phi$ is bounded from above, the sequence $(\phi(x^i_l ))$    converges and is therefore Cauchy. This entails that for any $\epsilon >0$, there exists $n_0 \in \mathbb{N}$ such that for any $q>p>n_0 $ , we have $\phi(x^i_q ) - \phi(x^i_p ) < \epsilon$. Since whenever $q>p>n_0,$  $x^i_p  \preceq x^i_q $ it follows that
	
	
	\[
	G(x^i_p,x^i_q,x^i_q) \leq \phi(x^i_q ) - \phi(x^i_p )  < \epsilon.
	\]
	
	
	We conclude that $(x^i_l )$ is a $G$-Cauchy in the complete space $(X,G)$, hence there exists $x^{*,i} \in X$ such that $ x^i_l \longrightarrow x^{*,i}$. Since $F$ is sequentially continuous, we have
	
	\begin{align*}
	x^i_{l+1} \longrightarrow x^{*,i} & \Longleftrightarrow x^i_{l+1}=F(x_l^i,x_l^{i+1},\cdots,x_l^n,x_l^1,\cdots,x_l^{i-1})\longrightarrow x^{*,i} \\
	& \Longleftrightarrow F(x^{*,i},x^{*,i+1},\cdots,x^{*,n},x^{*,1},\cdots,x^{*,i-1})= x^{*,i}.
	\end{align*}
	
	Thus we have proved that $(x^{*,1},\cdots,x^{*,n})$ is a $n$-tuple fixed point of $F$.
	
\end{proof}

\begin{example}
	
	We take $n=3.$
	Let $X = [0, \infty)$ and $G(x, y,z) = \max\{|x - y|,|x-z|,|y-z|\}$, then $(X, G)$ is
	a complete $G$-metric space and "$\leq$" is the ordering induced by $\phi(x)=2x$. Let $F : X \times X \to X$ be defined as follows:
	$$F (x, y,z) = x(1 + y)(2+z)$$
	and $F$ is obviously a non-decreasing function on $X$.
	
	If we let $x_0 = 1$ and $y_0 =z_0 =0$ then $$F (x_0 , y_0,z_0 ) = 1 · (1 + 0)(2+0) = 2, \
	F (y_0 ,z_0, x_0 ) = 0 · (1 + 0)(2+1) = 0$$
	 and $$F (z_0 ,x_0, y_0 ) = 0 · (1 + 1)(2+0)=0.$$
	
So we see that $$x_0 \leq F (x_0 , y_0,z_0 ), \ y_0 \leq F (y_0 ,z_0, x_0 ) \ \text{ and }\  z_0\leq F (z_0 ,x_0, y_0 ).$$ 

Also $$F (0, y,z) =
0 · (1 + y)(2+z) = 0, \ F (0,z,x) = 0 · (1 + z)(2+x) = 0,$$   and  $$F(0,x,y)=0 .(1+x)(2+y)=0.$$
Hence $(0, 0,0)$ is a $3$-tuple
fixed point of F.	
\end{example}

\section{Common $n$-tuple fixed point}

Now we define the concept of \textit{weakly related} mappings on preordered spaces as follows:

\begin{definition}(See \cite{gaba})
	Let $(X,\preceq)$ be a preordered space, and $F:X^n\to X$ and $g:X\to X$ be two mappings. Then the pair $\{F,g\}$ is said to be \textbf{weakly related} if 
	the following condition is satisfied:
	\begin{enumerate}
		
		\item[(C1)] 
		$$F(x^i,x^{i+1},\cdots,x^n,x^1,\cdots,x^{i-1})\preceq gF(x^i,x^{i+1},\cdots,x^n,x^1,\cdots,x^{i-1}) $$ and $$gx^i \preceq F(gx^i,gx^{i+1},\cdots,gx^n,gx^1,\cdots,gx^{i-1})$$ for all $1\leq i \leq n$.
		

	\end{enumerate}
	
\end{definition}
Now we state and prove the first common $n$-tuple fixed point existence theorem for the weakly related mappings.

\begin{theorem}\label{theorem2}
	Let $(X, G)$ be a complete $G$-metric space, $\phi : X \rightarrow \mathbb{R}$ be a bounded from above function and ``$\preceq$" the preorder induced by $\phi$. Let $F : X^n \rightarrow X, n\geq 2$ and $g:X\to X$ be two sequentilally continuous mapping on $X$ such that the pair $\{F,g\}$ is weakly related. If there exist $n$ elements $x_0^1,\cdots,x_0^n \in X$ with 
	
	\[
	x_0^i \preceq F(x_0^i,x_0^{i+1},\cdots,x_0^n,x_0^1,\cdots,x_0^{i-1})\  \text{  for all } i, \ 1\leq i \leq n.
	\]
	Then $F$ and $g$ have a common $n$-tuple fixed point in $X^n$.
	
\end{theorem}

\begin{proof}

	Let $x_0^1,\cdots, x_0^n\in X$  
	\begin{equation}\label{eq7}
	x_0^i \preceq F(x_0^i,x_0^{i+1},\cdots,x_0^n,x_0^1,\cdots,x_0^{i-1})\ \ \text{  for all } i, \ 1\leq i \leq n.
	\end{equation}
	
	We construct the sequences $(x_l^i)_l$ for $1\leq i \leq n$ as follows:
	
	\begin{equation}
	x_{2l+1}^i = F(x_{2l}^i,x_{2l}^{i+1},\cdots,x_{2l}^n,x_{2l}^1,\cdots,x_{2l}^{i-1})
	\end{equation}
	and
	$$  x_{2l+2}^i  = gx_{2l+1}^i ,$$
	for all $l\geq 0.$
	
	We shall show that
	\begin{equation}\label{eq9}
	x_l^i \preceq x_{l+1}^i \ \ \  \text{  for all } l\geq 0.
	\end{equation} 
	
	Since $x_0^i \preceq F(x_0^i,x_0^{i+1},\cdots,x_0^n,x_0^1,\cdots,x_0^{i-1}) $, using \eqref{eq7}, we have $x_0^i \preceq x_1^i.$ Again since the pair $\{F,g\}$ is weakly related, we have
	
	\begin{align*}
	x_1^i = F(x_0^i,x_0^{i+1},\cdots,x_0^n,x_0^1,\cdots,x_0^{i-1}) &\preceq gF(x_0^i,x_0^{i+1},\cdots,x_0^n,x_0^1,\cdots,x_0^{i-1})\\
	&= gx^i_1 =x^i_2,
	\end{align*}
	i.e $$x_1^i \preceq x^i_2 .$$
	
	Also, since $gx_1^i\preceq F(gx_1^i,gx_1^{i+1},\cdots,gx_1^n,gx_1^1,\cdots,gx_1^{i-1}),$ we have

	\begin{align*}
	x^i_2 = gx^i_1 &\preceq F(gx_1^i,gx_1^{i+1},\cdots,gx_1^n,Gx_1^1,\cdots,gx_1^{i-1}) \\
	&=  F(x_2^i,x_2^{i+1},\cdots,x_2^n,x_2^1,\cdots,x_2^{i-1})  = x_3^i,
	\end{align*}
	
	i.e $$x_2^i \preceq x^i_3 .$$

Thus by mathematical induction we conclude that \eqref{eq9} holds for all $l \geq 0$. Therefore 
\[
x_0^i \preceq x_1^i \preceq x_2^i \preceq x_3^i \preceq \cdots \preceq x_l^i \preceq \cdots .
\]
	
	By definition of the preorder, we have
	
	\[
	\phi(x_0^i) \leq \phi(x_1^i) \leq \phi(x_2^i) \leq \phi(x_3^i) \leq \cdots \leq \phi(x_l^i) \leq \cdots .
	\]
	
	Hence, the sequence $(\phi(x^i_l ))$ is a non-decreasing sequence of real numbers. Since $\phi$ is bounded from above, the sequence $(\phi(x^i_l ))$    converges and is therefore Cauchy. This entails that for any $\epsilon >0$, there exists $n_0 \in \mathbb{N}$ such that for any $q>p>n_0 $ , we have $\phi(x^i_q ) - \phi(x^i_p ) < \epsilon$. Since whenever $q>p>n_0,$  $x^i_p  \preceq x^i_q $ it follows that
	
	
	\[
	G(x^i_p,x^i_q,x^i_q) \leq \phi(x^i_q ) - \phi(x^i_p )  < \epsilon.
	\]

	We conclude that $(x^i_l )$ is a $G$-Cauchy sequence in $X$ and since $X$ is $G$-complete space, there exists $x^{*,i} \in X$ such that $ x^i_l \longrightarrow x^{*,i}$. 
	
	Since $F$ and $g$ are sequentially continuous, it is easy to see that
	
	\begin{align*}
	x^i_{2l+1} \longrightarrow x^{*,i} & \Longleftrightarrow x^i_{2l+1}=F(x_{2l}^i,x_{2l}^{i+1},\cdots,x_{2l}^n,x_{2l}^1,\cdots,x_{2l}^{i-1})\longrightarrow x^{*,i} \\
	& \Longleftrightarrow F(x^{*,i},x^{*,i+1},\cdots,x^{*,n},x^{*,1},\cdots,x^{*,i-1})= x^{*,i},
	\end{align*}
	and
	
	$$   x^i_{2l+2} \longrightarrow x^{*,i}  \Longleftrightarrow          x^i_{2l+2} = gx^i_{2l+1}  \longrightarrow x^{*,i} \Longleftrightarrow  gx^{*,i}= x^{*,i}.     $$

	and hence $$gx^{*,i}= x^{*,i}=F(x^{*,i},x^{*,i+1},\cdots,x^{*,n},x^{*,1},\cdots,x^{*,i-1}). $$
	Hence $(x^{*,1},\cdots,x^{*,n})$ is a common $n$-tuple fixed point of $F$ and $g$.
\end{proof}

\begin{example}
Let $X = [0, \infty)$ and $G(x, y,z) = \max\{|x - y|,|x-z|,|y-z|\}$, then $(X, G)$ is
a complete $G$-metric space.
For any positive real number $a$, let $\phi_a : X \rightarrow \mathbb{R}$ be defined by $\phi_a (x) = ax$, and $\preceq$ be the preorder induced by $\phi_a $. We define $F : X^n \rightarrow X$ and $g: X \rightarrow X$ as
follows

\[
F(x^1,x^2,\cdots,x^n) = x^1 + |\sin(x^1x^2\cdots, x^n)| \text{ and } gx= 5x
\]

If we let $x_0^1=1$ and $x_0^i=0$ for $i=2,\cdots,n$ then $F(x_0^1,x_0^2,x_0^n)=1+0=1$ and $F(x_0^i,x_0^{i+1},\cdots,x_0^n,x_0^1,\cdots,x_0^{i-1})=0$ for $i=2,\cdots,n.$

So
$$x_0^i \preceq F(x_0^i,x_0^{i+1},\cdots,x_0^n,x_0^1,\cdots,x_0^{i-1})\ \ \text{  for all } i, \ 1\leq i \leq n.$$

We have on one hand 
\[
gF(x^i,x^{i+1},\cdots,x^n,x^1,\cdots,x^{i-1}) =5(x^i+|\sin(x^1x^2\cdots, x^n)|),
\]
i.e.

\[
F(x^i,x^{i+1},\cdots,x^n,x^1,\cdots,x^{i-1}) \preceq gF(x^i,x^{i+1},\cdots,x^n,x^1,\cdots,x^{i-1}),
\]

and on the other hand,

\begin{align*}
F(gx^i,gx^{i+1},\cdots,gx^n,gx^1,\cdots,gx^{i-1}) &= F(5x^i,5x^{i+1},\cdots,5x^n,5x^1,\cdots,5x^{i-1})\\
&= 5x^i + |\sin(5^n x^1x^2\cdots, x^n)|,
\end{align*}

i.e.

$$ gx^i \preceq F(gx^i,gx^{i+1},\cdots,gx^n,gx^1,\cdots,gx^{i-1}).$$
And so the pair $\{F, g\}$ is weakly related. Again, it is not hard to see that $F$ and $g$ are sequentially continuous mappings. Hence we see that all the conditions of Theorem \ref{theorem2} are satisfied. Also we have

\[
F(0,x^i,x^{i+1},\cdots,x^n,x^1,\cdots,x^{i-1})=0
\]
for $i=1,\cdots n$ and $g(0)=0$. Thus $\underset{n}{\underbrace{(0,\cdots,0)}}$ is a common $n$-tuple fixed point for $F$ and $g$. 

\end{example}

Now, we present a result on $n$-tuple fixed point for a family of three maps.

\begin{theorem}\label{theorem3}
Let $(X, d)$ be a complete $G$-metric space, $\phi : X \rightarrow \mathbb{R}$ be a bounded from above function and ``$\preceq$" the preorder induced by $\phi$. Let $F : X^n \rightarrow X,$ $n\geq 2$ and $G,H:X\to X$ be three sequentilally continuous mapping on $X$ such that the pairs $\{F,G\}$ and $\{F,H\}$ are weakly related. Then $F,H$ and $G$ have a $n$-tuple fixed point.
\end{theorem}

\begin{proof}
Let $x_0^1,\cdots, x_0^n\in X$. We construct the sequences $(x_l^i)_l$ in $X$ as follows:

$$  Hx_{3l-3}^i  = x_{3l-2}^i , \qquad x_{3l}^i = G x_{3l-1}^i $$

and 
$$ x_{3l-1}^i = F(x^i_{3l-2},x^{i+1}_{3l-2}\cdots,x^n_{3l-2},x^1_{3l-2},\cdots,x^{i-1}_{3l-2})  $$

for all $l\geq 1.$
We shall show that 
\begin{equation}\label{eq10}
x_l^i \preceq x_{l+1}^i \ \ \  \text{  for all } l\geq 0.
\end{equation}

We have $x_1^i= Hx_0^i $. Since the pair $\{F,H\}$ is weakly related, we have $$x_1^i= Hx_0^i \preceq F(Hx_0^i,Hx_0^{i+1},\cdots,Hx_0^n,Hx_0^1,\cdots,Hx_0^{i-1})=x_2^i.$$ 

Again since the pair $\{F,G\}$ is weakly related, we have 
\begin{align*}
x_2^i &= F(Hx_0^i,Hx_0^{i+1},\cdots,Hx_0^n,Hx_0^1,\cdots,Hx_0^{i-1})\\ &\preceq GF(Hx_0^i,Hx_0^{i+1},\cdots,Hx_0^n,Hx_0^1,\cdots,Hx_0^{i-1})=Gx_2^i =  x_3^i.
\end{align*}
Similarly, using repeatedly the fact that the pairs $\{F, G\}$ and $\{F, H\}$ are weakly related, we get

\[
x_0^i \preceq x_1^i \preceq x_2^i \preceq x_3^i \preceq \cdots \preceq x_l^i \preceq \cdots .
\]

By definition of the preorder, we have

\[
\phi(x_0^i) \preceq \phi(x_1^i) \preceq \phi(x_2^i) \preceq \phi(x_3^i) \preceq \cdots \preceq \phi(x_l^i) \preceq \cdots .
\]

Hence, the sequence $(\phi(x^i_l ))$ is a non-decreasing sequence of real numbers. Since $\phi$ is bounded from above, the sequence $(\phi(x^i_l ))$    converges and is therefore Cauchy. This entails that for any $\epsilon >0$, there exists $n_0 \in \mathbb{N}$ such that for any $q>p>n_0 $ , we have $\phi(x^i_q ) - \phi(x^i_p ) < \epsilon$. Since whenever $q>p>n_0,$  $x^i_q  \preceq x^i_p $ it follows that

\[
d(x^i_p,x^i_x,x_q) \leq \phi(x^i_q ) - \phi(x^i_p )  < \epsilon.
\]

We conclude that $(x^i_l )$ is a $G$-Cauchy sequence in $X$ and since $X$ is $G$-complete, there exists $x^{*,i} \in X$ such that $ x^i_l \longrightarrow x^{*,i}$. Since $F,G$ and $H$ are sequentially continuous, it is easy to see that

\begin{align*}
x^i_{3l-1} \longrightarrow x^{*,i} & \Longleftrightarrow F(x_{3l-2}^i,x_{3l-2}^{i+1},\cdots,x_{3l-2}^n,x_{3l-2}^1,\cdots,x_{3l-2}^{i-1})\longrightarrow x^{*,i} \\
& \Longleftrightarrow F(x^{*,i},x^{*,i+1},\cdots,x^{*,n},x^{*,1},\cdots,x^{*,i-1})= x^{*,i},
\end{align*}
and

$$   x^i_{3l} \longrightarrow x^{*,i}  \Longleftrightarrow          x^i_{3l} = Gx^i_{3l-1}  \longrightarrow x^{*,i} \Longleftrightarrow  Gx^{*,i}= x^{*,i},    $$

and 

$$   x^i_{3l-2} \longrightarrow x^{*,i}  \Longleftrightarrow          x^i_{3l-2} = Hx^i_{3l-3}  \longrightarrow x^{*,i} \Longleftrightarrow  Hx^{*,i}= x^{*,i}.     $$

Therefore $$Hx^{*,i}=Gx^{*,i}= x^{*,i}=F(x^{*,i},x^{*,i+1},\cdots,x^{*,n},x^{*,1},\cdots,x^{*,i-1}). $$

Hence $(x^{*,1},\cdots,x^{*,n})$ is a common $n$-tuple fixed point of $F,G$ and $H$.

\end{proof}

\begin{example}
Let $X = [0, \infty)$ and $d(x, y,z) =\max\{ |x - y|,|x-z|,|y-z|\}$, then $(X, d)$ is
a complete $G$-metric space. For any positive real number $a$, let $\phi_a : X \rightarrow \mathbb{R}$ be defined by $\phi_a (x) = ax$, and $\preceq$ be the preorder induced by $\phi_a $. We define $F : X^n \rightarrow X$, $G: X \rightarrow X$ and $H: X \rightarrow X$ as
follows

\[
F(x^1,x^2,\cdots,x^n) = x^1 + |\sin(x^1x^2\cdots, x^n)|,\ Gx=5x \text{ , and } Hx= 6x.
\]

The pairs $\{F, G\}$ and $\{F,H\}$ are weakly related. Again, it is not hard to see that $F,G$ and $G$ are $d$-sequentially continuous mappings on $X$. Hence we see that all the conditions of our theorem are satisfied. Also we have

\[
F(0,x^i,x^{i+1},\cdots,x^n,x^1,\cdots,x^{i-1})=0
\]
for $i=1,\cdots n$ and $G(0)=0=H(0)$. Thus $\underset{n}{\underbrace{(0,\cdots,0)}}$ is a common $n$-tuple fixed point for $F, G$ and $H$.

\end{example}

Before we state our last result, we give the following definition:

\begin{definition}(Compare \cite[Definition 2.4]{gabao}) 
Let $(X,\preceq)$ be a preordered set and $g,f : X \to X$. We say that the pair $\{g,f\}$ (in this order) is an \textbf{embedded pair} if
\[
g(x) \preceq f(g(x)),  \text{ whenever } \ x\in X.
\]

\vspace*{0.3cm}

We shall say that the family $\{G_1, G_2, \cdots,G_n \}$ (in this order) is a \textbf{$n$-embedded chain} if 
for all $i=1,\cdots, n-1$, the pair $\{G_i,G_{i+1}\}$ is an embedded pair. Observe that an embedded pair is a $2$-embedded chain.

\vspace*{0.3cm}

We shall say that the family $\{G_1, G_2, \cdots,G_n \}$ is a \textbf{dual $n$-embedded chain} if $\{G_1, G_2, \cdots,G_n \}$ and $\{G_n, G_{n-1}, \cdots,G_1 \}$ are $n$-embedded chains.
\end{definition}

\begin{example}(Compare \cite[Example 2.5]{gabao})
Let $X=[2,\pi)$ with the usual order and consider the pairs $\mathcal{F}= \{ F_1(x)=3x, F_2(x)=5x\}$ and $\mathcal{G}=\{ G_1(x)=\sin x +1, G_2(x)=x^2\}$.

\vspace*{0.3cm}
For any $x\in X,$
$$F_1(x)=3x \leq 5(3x)= F_2(F_1(x))\text { and } F_2(x)=5x \leq 3(5x)= F_1(F_2(x)),$$ showing that $\mathcal{F}$ is a dual $2$-embedded chain.

\vspace*{0.3cm}

On the other way around $$x\in X, G_1(x)= \sin x +1 \leq (\sin x +1 )^2= G_2(G_1(x)),$$ showing that $\mathcal{G}$ is an embedded pair, while $$ G_2(x)=x^2 > \sin (x^2) + 1= G_1(G_2(x)),$$ showing that $\mathcal{G}$ is not a dual $2$-embedded chain.
\end{example}

Using the same approach as suggested in the proof of Theorem \ref{theorem3}, one can easily establish that:

\begin{theorem}\label{theorem4}
	Let $(Y, d)$ be a complete $G$-metric space, $\phi : X \rightarrow \mathbb{R}$ be a bounded from above function and ``$\preceq$" the preorder induced by $\phi$. Let $F: Y^n \to Y$ and $G_i : Y \to Y; i= 2, \cdots ,r$ for $r> 2$ be $(r-1)+1$  sequentially continuous mapping on $Y$ such that
	the pairs $\{F, G_r\}; r= 2, 3 $ are weakly related. Moreover, we assume that $\{G_r, G_{r-1}, \cdots,G_3\}$ is an $r-2$-embedded chain.
	Then $F,G_2,\cdots, G_r $ have a common n-tuple fixed point in $Y$.
\end{theorem}

\begin{proof}
	We give here a sketch of the proof.
	
Let $X_0^1,\cdots, X_0^n\in X$. Observe that the sequences $(X_l^i)_l$ in $Y$ constructed as follows:

\[
G_rX^i_{rl-r} = X^i_{rl-r+1}, \cdots , G_3X^i_{rl-3}=X^i_{rl-2}, \ G_2X^i_{rl-1}=X^i_{rl}
\]
and 
\[
X^i_{rl-1} = F(X^i_{rl-2},X^{i+1}_{rl-2},\cdots,X^n_{rl-2},X^1_{rl-2},\cdots,X^{i-1}_{rl-2}),
\]
for all $l\geq 1.$
are Cauchy and and since $Y$ is complete, there exist $X^{*,i}\in Y$ such that $X_l^i \longrightarrow X^{*,i} $.

Since $F$ and $G_2, \cdots, G_r$ are sequentially continuous, it is easy to see that

\begin{align*}
X_{rl-1}^i \longrightarrow X^{*,i}  & \Longleftrightarrow F(X_{rl-2}^i,X_{rl-2}^{i+1},\cdots,X_{rl-2}^n,X_{rl-2}^1,\cdots,X_{rl-2}^{i-1}) \longrightarrow X^{*,i} \\
&  \Longleftrightarrow F(X^{*,i},X^{*,i+1},\cdots,X^{*,n}, X^{*,1}\cdots,X^{*,i-1})=X^{*,i}
\end{align*}

and
\[
X_{rl-r}^i \longrightarrow X^{*,i} \Longleftrightarrow X_{rl-k+1}^i=G_kX_{rl-k}^i \longrightarrow X^{*,i}  \Longleftrightarrow G_kX^{*,i}=X^{*,i},
\]
and hence 
\[
G_k X^{*,i}= X^{*,i} = F(X^{*,i},X^{*,i+1},\cdots,X^{*,n}, X^{*,1}\cdots,X^{*,i-1}).
\]

Hence $(X^{*,1},X^{*,2},\cdots,X^{*,n})$ is a common n-tuple fixed point of $F$ and $G_2,\cdots, G_r$.

\end{proof}

\begin{example}(Compare \cite[Example 2.7]{gabao})
	
Let $X = [0, \infty)$ and $d(x, y,z) =\max\{ |x - y|,|x-z|,|y-z|\}$, then $(X, d)$ is
	a complete $G$-metric space. For any positive real number $a$, let $\phi_a: X\to \mathbb{R}$ be defined by $\phi_a(x)=ax$, and $\preceq$ the preorder induced by $\phi_a$. We define $F:X^n\to X$ and $G:X\to X$ as follows
	$$F (x^1,x^2,\cdots,x^n) = x^1 + |\sin(x^1x^2\cdots x^n)| \text{ and } G_k(x) = kx , k=2,\cdots, r, \ r>2.$$

	For $k=1,2$ , we have on one hand, 
	
	$$G_kF(x^i,x^{i+1},\cdots,x^n, x^1,\cdots,x^{i-1})=k(x^i + |\sin(x^1x^2\cdots x^n)|),$$ i.e. $$F(x^i,x^{i+1},\cdots,x^n, x^1,\cdots,x^{i-1}) \preceq G_kF(x^i,x^{i+1},\cdots,x^n, x^1,\cdots,x^{i-1}),$$ 
	
	and on the other hand,
	
	\begin{align*}
	F(G_kx^i,G_kx^{i+1},\cdots,G_kx^n, G_kx^1,\cdots,G_kx^{i-1})&= 
	F(kx^i,kx^{i+1},\cdots,kx^n, kx^1,\cdots,kx^{i-1})\\
	& =kx^i+|\sin(k^nx^1x^2\cdots x^n)|, 
	\end{align*}
	i.e. $$G_kx^i\preceq F(G_kx^i,G_kx^{i+1},\cdots,G_kx^n, G_kx^1,\cdots,G_kx^{i-1}).$$
	
	And so the pair $\{F,G_k\}$ are weakly-related for $k=2,3$. Again, it is not hard to see that $F$ and $G_k, k=2,\cdots, r,$ are sequentially continuous mappings on $X$.
	
	Moreover, for any $x\in[0,\infty), \ kx \leq k(k-1) x,  k=2,\cdots, r,$ implying that 
	$\{G_r, G_{r-1}, \cdots,G_3\}$ is an an $r-2$-embedded chain.
	Hence we see that all the conditions of our theorem are satisfied. 
	
	Also we have $$F (0, x^i,x^{i+1},\cdots, x^n,x^1,\cdots,x^{i-2}) = 0= G_k(0)$$ for $k=2,\cdots,r$ and for $i=1,2,\cdots,n$. 
	
	Thus $\underbrace{(0,\cdots,0)}_\text{n}$ is a common n-tuple fixed point of $F,G_2,\cdots, G_r $.

\end{example}

\bibliographystyle{amsplain}

\end{document}